\documentclass[12pt,a4paper]{amsart}

\usepackage{amsmath, amsthm,amssymb,tikz,pgflibraryarrows}
\usepackage{latexsym}		
\usepackage{graphicx}		
\usepackage{enumerate}
\usepackage{fancyhdr}

\newtheorem{thm}{Theorem}[section]

\newtheorem{lemma}[thm]{Lemma}
\newtheorem{prop}[thm]{Proposition}
\newtheorem{cor}[thm]{Corollary}
\newtheorem*{prop*}{Proposition}
\newtheorem*{lemma*}{Lemma}

\theoremstyle{remark}
\newtheorem{example}[thm]{Example}
\newtheorem{remark}[thm]{Remark}

\def\N{{\mathbb N}}

\def\H{{\mathcal H}}
\def\K{{\mathcal K}}

\def\Z{{\mathbb Z}}

\newcommand{\clsp}{\overline{\operatorname{span}}}
\newcommand{\lsp}{\operatorname{span}}
\newcommand{\End}{\operatorname{End}}

\newcommand{\range}{\operatorname{range}}

\newcommand{\lt}{\operatorname{lt}}

\numberwithin{equation}{section}
\numberwithin{thm}{section}

\begin{document}

\title[Higher-rank graphs]{Skew-products of higher-rank graphs\\ and crossed products by semigroups}

\author{Ben Maloney}

\address{Ben Maloney, School of Mathematics and Applied Statistics,
University of Wollongong, NSW 2522, Australia}
\email{bkm611@uowmail.edu.au}

\author{David Pask}
\address{David Pask, School of Mathematics and Applied Statistics,
University of Wollongong, NSW 2522, Australia}
\email{dpask@uow.edu.au}

\author[Iain Raeburn]{Iain Raeburn}

\address{Iain Raeburn, Department of Mathematics and Statistics, University of Otago, PO Box~56, Dunedin 9054, New Zealand}
\email{iraeburn@maths.otago.ac.nz}

\date{\today}
\thanks{This research was supported by the Australian Research Council and the University of Otago.}

\begin{abstract}
We consider a free action of an Ore semigroup on a higher-rank graph, and the induced action by endomorphisms of the $C^*$-algebra of the graph. We show that the crossed product by this action is stably isomorphic to the $C^*$-algebra of a quotient graph. Our main tool is Laca's dilation theory for endomorphic actions of Ore semigroups on $C^*$-algebras, which embeds such an action in an automorphic action of the enveloping group on a larger $C^*$-algebra.
\end{abstract}

\maketitle

\section{Introduction}

Kumjian and Pask \cite{KP1} proved that if a group $G$ acts freely on a directed graph $E$, then the associated crossed product $C^*(E)\rtimes G$ of the graph algebra is stably isomorphic to the graph algebra $C^*(G\backslash E)$ of the quotient graph. Their theorem has been extended in several directions: to actions of groups on higher-rank graphs (\cite[Theorem~5.7]{KP2} and \cite[Corollary~7.5]{PQR}), and to actions of Ore semigroups on directed graphs \cite{PRY}. Here we consider actions of Ore semigroups on higher-rank graphs.

Our main theorem directly extends that of \cite{PRY} to higher-rank graphs, but our proof has some interesting new features. First of these is our more efficient use of Laca's dilation theory for endomorphic actions \cite{Laca}: by exploiting his uniqueness theorem, we have been able to bypass the complicated direct-limit constructions used in \cite{PRY}. Second, we have found an explicit isomorphism. In searching for explicit formulas, we have revisited the case of group actions, and we think a third feature of general interest is our direct approach to crossed products of the $C^*$-algebras of skew-product graphs, which is based on the treatment of skew products of directed graphs in \cite[\S3]{KQR} (see Theorem~\ref{cpofskewp}).

After a brief discussion of notation and background material, we discuss higher-rank graphs and their $C^*$-algebras in \S\ref{sec-hrgraphs}, and prove two general lemmas about the $C^*$-algebras of higher-rank graphs. In \S\ref{sec-GT}, we prove our first results about actions of semigroups, including a version of the Gross-Tucker theorem which will allow us to replace the underlying graph with a skew product. In \S\ref{sec-dilate} we apply Laca's dilation theory to higher-rank graph algebras. In \S\ref{sec-group} we discuss group actions on skew products, and then in \S\ref{sec-main} we pull the pieces together and prove our main theorem.

\subsection*{Background and notation}

All semigroups in this paper are countable and have an identity $1$. An \emph{Ore semigroup} is a cancellative semigroup such that for all pairs $t,u\in S$, there exist $x,y\in S$ such that $xt=yu$. Ore and Dubreil proved that a semigroup is Ore if and only if it can be embedded in a group $\Gamma$ such that $\Gamma=S^{-1}S$; the group $\Gamma$ is unique up to isomorphism, and we call it the \emph{enveloping group} of $S$.

An action of a semigroup $S$ on a $C^*$-algebra $A$ is an identity-preserving homomorphism $\alpha$ of $S$ into the semigroup $\End A$ of endomorphisms of $A$. A \emph{covariant representation} of $(A,S,\alpha)$ in in a $C^*$-algebra $B$ consists of a nondegenerate homomorphism $\pi\to B$ and a homomorphism $V$ of $S$ into the semigroup of isometries in $M(B)$, such that $\pi(\alpha_t(a))=V_t\pi(a)V_t^*$ for $a\in A$ and $t\in S$. The \emph{crossed product} $A\times_\alpha S$ is generated by a universal covariant representation $(i_A,i_S)$ in $A\times_\alpha S$. (In the recent literature, this is called the ``Stacey crossed product''.)  When $S=\Gamma$, the endomorphisms are automorphisms, and we recover the usual crossed product $(A\rtimes_\alpha \Gamma,i_A,i_\Gamma)$. If $(\pi,V)$ is a covariant representation of $(A,S,\alpha)$ in $B$, then we write $\pi\times V$ for the homomorphism of $A\times_\alpha S$ into $B$ such that $(\pi\times V)\circ i_A=\pi$ and $(\pi\times V)\circ i_S=V$.

To talk about stable isomorphisms, we need to consider tensor products with the algebra $\K(\H)$ of compact operators. Since $\K(\H)$ is nuclear, there is no ambiguity in writing $A\otimes \K(\H)$. However, we are interested in $C^*$-algebras which have universal properties, and we view $A\otimes \K(\H)$ as the maximal tensor product $A\otimes_{\max} \K(\H)$ which is universal for pairs of commuting representations of $A$ and $\K(\H)$ (see \cite[Theorem~B.27]{RW}).

We write $\lambda$ and $\rho$ for the left- and right-regular representations of a group $\Gamma$ on $l^2(\Gamma)$, and $\{e_g:g\in\Gamma\}$ for the usual orthonormal basis of point masses. For $F\subset\Gamma$, $\chi_F$ is the operator on $l^2(\Gamma)$ which multiplies by the characteristic function of $F$, and $\chi_g:=\chi_{\{g\}}$. We often use the relations $\lambda_h\chi_g=\chi_{hg}\lambda_h$ and $\rho_k\chi_g=\chi_{gk^{-1}}\rho_k$. When $S$ is a subsemigroup of $\Gamma$, we identify $l^2(S)$ with the subspace $\clsp\{e_t:t\in S\}$ of $l^2(\Gamma)$, and then $t\mapsto \lambda^S_t:=\lambda_t|_{l^2(S)}$ is the usual Toeplitz representation of $S$ on $l^2(S)$.

\section{Higher-rank graphs and their $C^*$-algebras}\label{sec-hrgraphs}

Suppose $k\in\N$ and $k\geq 1$. A \emph{graph of rank $k$}, or \emph{$k$-graph}, is a countable category $\Lambda$ with domain and codomain maps $r$ and $s$, together with a functor $d:\Lambda\to \N^k$ satisfying the \emph{factorisation property}: for every $\lambda\in\Lambda$ and decomposition $d(\lambda)=m+n$ with $m,n\in \N^k$, there is a unique pair $(\mu,\nu)$ in $\Lambda\times \Lambda$ such that $s(\mu)=r(\nu)$, $d(\mu)=m$, $d(\nu)=n$ and $\lambda=\mu\nu$. We write $\Lambda^0$ for the set of objects, and observe that the factorisation property allows us to identify $\Lambda^0$ with $d^{-1}(0)$; then we write $\Lambda^n:=d^{-1}(n)$ for $n\in \N^k$. Visualisations of $k$-graphs are discussed in \cite{RSY} and \cite[Chapter 10]{cbms}: we think of $\Lambda^0$ as the set of vertices, and $\lambda\in \Lambda^n$ as a path of degree $n$ from $s(\lambda)$ to $r(\lambda)$.

As in \cite{KP2}, we assume throughout that our $k$-graphs are \emph{row-finite} and \emph{have no sources}, in the sense that $v\Lambda^n:=r^{-1}(v)\cap \Lambda^n$ is finite and nonempty for every $v\in \Lambda^0$, $n\in\N^k$.

Given a $k$-graph $\Lambda$, a \emph{Cuntz-Krieger $\Lambda$-family} in a $C^*$-algebra $B$ consists of partial isometries $\{S_\lambda:\lambda\in\Lambda\}$ in $B$ satisfying
\begin{itemize}\item[]\begin{itemize}
\item[(CK1)] $\{S_v:v\in \Lambda^0\}$ are mutually orthogonal projections;
\item[(CK2)] $S_{\lambda}S_{\mu}=S_{\lambda\mu}$ whenever $s(\lambda)=r(\mu)$;
\item[(CK3)] $S_\lambda^*S_\lambda=S_{s(\lambda)}$ for every $\lambda\in\Lambda$;
\item[(CK4)] $S_v=\sum_{\lambda\in v\Lambda^n}S_\lambda S_\lambda^*$ for every $v\in \Lambda^0$ and $n\in \N^k$.
\end{itemize}\end{itemize}
The graph algebra $C^*(\Lambda)$ is generated by a universal Cuntz-Krieger $\Lambda$-family $\{s_\lambda\}$. When there is more than one graph around, we sometimes write $\{s_\lambda^\Lambda\}$ for emphasis. Each vertex projection $s_v$ (and hence by (CK3) each $s_{\lambda}$) is non-zero \cite[Proposition~2.11]{KP2}, and
\[
C^*(\Lambda)=\clsp\{s_\lambda s_\mu^*:\lambda,\mu\in\Lambda\}\ \text{ (see \cite[Lemma~3.1]{KP2}).}
\]

\begin{lemma}\label{defsumpis}
Suppose that $\Lambda$ is a row-finite $k$-graph with no sources, that $m\in \N^k$, and that $V$ is a subset of $\Lambda^m$ such that the paths in $V$ all have different sources. Let $\{F_n\}$ be an increasing sequence of finite subsets of $V$ such that $V=\bigcup_n F_n$. Then each $s_n:=\sum_{\mu\in F_n} s_\mu$ is a partial isometry, and there is a partial isometry $s_V\in M(C^*(\Lambda))$ such that $s_n\to s_V$ strictly. The limit $s_V$ is independent of the choice of $F_n$, and satisfies
\begin{equation}\label{leftmult}
s_Vs_\alpha s_\beta^*=
\begin{cases}
s_{\mu\alpha} s_\beta^*&\text{if $r(\alpha)=s(\mu)$ for some $\mu\in V$}\\
0&\text{otherwise,}
\end{cases}
\end{equation}
and, for paths $\beta$ with $d(\beta)\geq m$,
\begin{equation}\label{rightmult}
s_\alpha s_\beta^*s_V=
\begin{cases}
s_{\alpha} s_{\beta'}^*&\text{if $\beta=\mu\beta'$ for some $\mu\in V$}\\
0&\text{otherwise.}
\end{cases}
\end{equation}
If $V\subset \Lambda^m$ and $W\subset \Lambda^p$ are two such sets, then $s_Vs_W$ is the partial isometry $s_{VW}$ associated to the set $VW:=\{\mu\nu:\mu\in V,\ \nu\in W\text{ and }s(\mu)=r(\nu)\}$.
\end{lemma}

\begin{proof}
Since  all the $\mu$ have the same degree, (CK3) and (CK4) imply that
\[
s_n^*s_n=\sum_{\mu,\nu\in F_n} s_\mu^*s_\nu=\sum_{\mu\in F_n} s_\mu^*s_\mu = \sum_{\mu\in F_n}s_{s(\mu)};
\]
since $s(\mu)\not=s(\nu)$ for $\mu\not=\nu$ in $V$, this is a sum of mutually orthogonal projections, and hence is a projection. Thus $s_n$ is a partial isometry. For $\alpha,\beta\in \Lambda$, we have
\begin{equation}\label{leftmult2}
s_ns_\alpha s_\beta^*=
\begin{cases}
s_{\mu\alpha} s_\beta^*&\text{if $r(\alpha)=s(\mu)$ for some $\mu\in F_n$}\\
0&\text{otherwise.}
\end{cases}
\end{equation}
If $r(\alpha)=s(\mu)$ for some $\mu\in V$, then $\mu\in F_n$ for large $n$, and hence the right-hand side of \eqref{leftmult2} is eventually constant for every $s_\alpha s_\beta^*$. Now an $\epsilon/3$ argument implies that $\{s_na\}$ is Cauchy for every $a\in C^*(\Lambda)$. A similar calculation shows that $s_\alpha s_\beta^* s_n$ is eventually constant whenever $d(\beta)\geq m$. However, (CK4) and (CK2) imply that
\[
\lsp\big\{s_\alpha s_\beta^*:\alpha,\beta\in \Lambda\big\}
=\lsp\big\{s_\alpha s_\beta^*:\alpha,\beta\in \Lambda,\ d(\beta)\geq m\big\}
\]
so $s_\alpha s_\beta^*s_n$ is eventually constant for all $\alpha,\beta$, and we deduce as before that $\{as_n\}$ is Cauchy for all $a\in C^*(\Lambda)$. Since $M(C^*(\Lambda))$ is complete in the strict topology \cite[Proposition~3.6]{bus}, we deduce that $s_n$ converges strictly to a multiplier $s_V$. Then \eqref{leftmult2} implies \eqref{leftmult}, and similarly for \eqref{rightmult}.

The formula \eqref{leftmult} implies that $s_V$ is independent of the choice of sequence $\{F_n\}$. For $\alpha,\beta\in \Lambda$, \eqref{leftmult} and the adjoint of \eqref{rightmult} show that $s_Vs_V^*s_Vs_\alpha ^*s_\beta=0=s_Vs_\alpha s_\beta^*$ unless $r(\alpha)=s(\mu)$ for some $\mu\in V$, and in that case
\[
s_Vs_V^*s_Vs_\alpha ^*s_\beta=s_Vs_V^*s_{\mu\alpha}s_\beta^*=s_Vs_\alpha s_\beta^*;
\]
either way, we have $s_Vs_V^*s_Vs_\alpha s_\beta^*=s_{V}s_\alpha s_\beta^*$. Thus $s_V s_V^* s_V=s_V$, and $s_V$ is a partial isometry. The final assertion follows from two applications of \eqref{leftmult}.
\end{proof}

\begin{remark}
Lemma~\ref{defsumpis} applies when $m=0$, in which case the summands are projections and so is the limit $s_V$. To emphasise this, we write $p_V$ for $s_V$ when $m=0$.
\end{remark}

A $k$-graph morphism $\pi:\Lambda\to \Sigma$ is \emph{saturated} if $r(\sigma)\in \pi(\Lambda^0)\Longrightarrow \sigma\in \pi(\Lambda)$. Recall (from \cite{A}, for example) that a homomorphism $\phi$ from a $C^*$-algebra to a multiplier algebra $M(B)$ is \emph{extendible} if there are an approximate identity $\{e_i\}$ for $A$ and a projection $p\in M(B)$ such that $\phi(e_i)$ converges strictly to $p$. If so, there is a unique extension $\overline{\phi}:M(A)\to M(B)$, which satisfies $\overline{\phi}(1)=p$ and is strictly continuous. Nondegenerate homomorphisms are extendible with $\overline{\phi}(1)=1$.

\begin{lemma}\label{extendible}
Suppose that $\pi:\Lambda\to \Sigma$ is an injective saturated $k$-graph morphism between row-finite graphs with no sources. Then there is a homomorphism $\pi_*:C^*(\Lambda)\to C^*(\Sigma)$ such that $\pi_*(s^\Lambda_\lambda)=s^\Sigma_{\pi(\lambda)}$, and $\pi_*$ is injective and extendible with $\overline{\pi_*}(1)=p_{\pi(\Lambda^0)}$. The assignment $\pi\mapsto\pi_*$ is functorial: $(\pi\circ\tau)_*=\pi_*\circ\tau_*$.
\end{lemma}

\begin{proof}
Saturation means that $\{\sigma\in \Sigma:r(\sigma)=\pi(v)\}=\{\pi(\lambda):r(\lambda)=v\}$ for every $v\in \Lambda^0$, so the Cuntz-Krieger relation (CK4) in $\Sigma$ implies the analogous relation for the family $\{s^\Sigma_{\pi(\lambda)}:\lambda\in \Lambda\}$. Thus $\{s^\Sigma_{\pi(\lambda)}\}$ is a Cuntz-Krieger $\Lambda$-family, and there is a homomorphism $\pi_*$ satisfying $\pi_*(s^\Lambda_\lambda)=s^\Sigma_{\pi(\lambda)}$. Since $\pi$ is injective and every $s^\Sigma_{w}\not=0$, the gauge-invariant uniqueness theorem \cite[Theorem~3.4]{KP2} implies that $\pi_*$ is faithful.

To see that $\pi_*$ is extendible, write $\Lambda^0=\bigcup_n F_n$ as an increasing union of finite sets. Then $p_n:=\sum_{v\in F_n}s^{\Lambda}_v$ is an approximate identity for $C^*(\Lambda)$. The images $\pi(F_n)$ satisfy $\bigcup_n \pi(F_n)=\pi(\Lambda^0)$, and since $\pi$ is injective,
\[
\pi(p_n)=\sum_{v\in F_n}p^\Sigma_{\pi(v)}=\sum_{w\in \pi(F_n)}p^\Sigma_{w},
\]
which by Lemma~\ref{defsumpis} converge strictly to $p_{\pi(\Lambda^0)}$. Thus $\pi_*$ is extendible with $\overline{\pi_*}(1)=p_{\pi(\Lambda^0)}$. The functoriality follows from the formula $\pi_*(s^\Lambda_\lambda)=s^\Sigma_{\pi(\lambda)}$.
\end{proof}

\section{A Gross-Tucker theorem}\label{sec-GT}

Suppose that $\alpha$ is a left action of an Ore semigroup $S$ on a $k$-graph $\Sigma$, and that $\alpha$ is \emph{free} in the sense that $\alpha_t(\lambda)=\alpha_u(\lambda)$ implies $t=u$. (It suffices to check freeness on vertices.) We will show that if $\alpha$ admits a fundamental domain, then there is an isomorphism of $\Sigma$ onto a skew product which carries $\alpha$ into a canonical action of $S$ by left translation. Such results were first proved for actions of groups on directed graphs by Gross and Tucker (see, for example, \cite[Theorem~2.2.2]{GT}), and when $S$ is a group, Theorem~\ref{GTthm} below was proved by Kumjian and Pask \cite[Remark~5.6]{KP2}.

Even the first step, which is the construction of the quotient graph, relies on the Ore property. We define a relation $\sim$ on $\Sigma$ by
\[
\lambda\sim\mu\Longleftrightarrow \ \text{if there exist $t,u\in S$ such that $\alpha_t(\lambda)=\alpha_u(\mu)$.}
\]
The relation $\sim$ is trivially reflexive and symmetric. To see that it is transitive, suppose $\lambda\sim\mu$ and $\mu\sim\nu$, so that there exist $s,t,u,v\in S$ such that $\alpha_s(\lambda)=\alpha_t(\mu)$ and $\alpha_u(\mu)=\alpha_v(\nu)$. Since $S$ is Ore, there exist $x,y\in S$ such that $xt=yu$. Then
$\alpha_{xs}(\lambda)=\alpha_{xt}(\mu)=\alpha_{yu}(\mu)=\alpha_{yv}(\nu)$, which implies that $\lambda\sim\nu$. Thus $\sim$ is an equivalence relation on $\Sigma$. Since equivalent elements have the same degree, it makes sense to write $(S\backslash\Sigma)^0$ for the set of equivalence classes of vertices, $S\backslash\Sigma$ for the set of all equivalence classes, and to define $d:S\backslash \Sigma\to \N^k$ by $d([\lambda])=d(\lambda)$. It is easy to check that there are well-defined maps $r,s:S\backslash\Sigma\to (S\backslash\Sigma)^0$ such that $r([\lambda])=[r(\lambda)]$ and $s([\lambda])=[s(\lambda)]$.

\begin{lemma} \label{quotient}
With notation as above $((S\backslash \Sigma)^0,S\backslash \Sigma,r,s,d)$ is a $k$-graph, with composition defined by
\begin{equation}\label{defcomp}
[\lambda][\mu]=[\alpha_t(\lambda)\alpha_u(\mu)]\ \text{where $t,u\in S$ satisfy $\alpha_t(s(\lambda))=\alpha_u(r(\mu))$,}
\end{equation}
and $q:\lambda\mapsto [\lambda]$ is a $k$-graph morphism.
\end{lemma}

\begin{proof}
To verify that $S\backslash \Sigma$ is a $k$-graph, we have to check that:
\begin{itemize}
\item the right-hand side of \eqref{defcomp} is independent of the choice of $t$ and $u$ (this uses the Ore property and the freeness of the action);
\item the right-hand side of \eqref{defcomp} is independent of the choice of coset representatives: (this uses the Ore property);
\item $r([\lambda][\mu])=r([\lambda])$ and $s([\lambda][\mu])=s([\mu])$;
\item associativity (this uses the Ore property);
\item the classes $[\iota_v]$ have the properties required of the identity morphisms at $[v]$;
\item $S\backslash \Sigma$ has the factorisation property.
\end{itemize}
Finally, if $\lambda$ and $\mu$ are composable, we can take $t=u=1$ in \eqref{defcomp}, and deduce that $q(\lambda\mu)=q(\lambda)q(\mu)$.
\end{proof}

Now suppose that $\Lambda$ is a $k$-graph and $\eta:\Lambda\to S$ is a functor into a semigroup $S$ (viewed as a category with one object). As in \cite[Definition~5.1]{KP2}, we can make the set-theoretic product $\Lambda\times S$ into a $k$-graph $\Lambda\times_\eta S$ by taking $(\Lambda\times_\eta S)^0=\Lambda^0\times S$, defining $r,s:\Lambda\times_\eta S\to (\Lambda\times_\eta S)^0$ by
\[
r(\lambda,t)=(r(\lambda),t)\ \text{ and }\ s(\lambda,t)=(s(\lambda),t\eta(\lambda)),
\]
defining the composition by
\[
(\lambda,t)(\mu, u)=(\lambda\mu,t)\ \text{ when $s(\lambda,t)=r(\mu, u)$ (which is equivalent to $u=t\eta(\lambda)$\,),}
\]
and defining $d:\Lambda\times_\eta S\to \N^k$ by $d(\lambda,t)=d(\lambda)$. Of course, one has to check the axioms to see that this does define a $k$-graph, but this is routine. We call $\Lambda \times_\eta S$ a \emph{skew product}. Every $\Lambda\times_\eta S$ carries a natural action $\lt$ of $S$ defined by $\lt_u(\lambda,t)=(\lambda,ut)$, and this action is free because $S$ is cancellative.  

The Gross-Tucker theorem implicit in \cite[Remark~5.6]{KP2} says that every free action of a group $\Gamma$ on a $k$-graph $\Sigma$ is isomorphic to the action $\lt$ on a skew-product $(\Gamma\backslash\Sigma)\times_\eta\Gamma$. As in \cite{PRY}, to get a Gross-Tucker theorem for actions of an Ore semigroup $S$, one has to insist that the action admits a \emph{fundamental domain}, which is a subset $F$ of $\Sigma$ such that for every $\sigma\in\Sigma$ there are exactly one $\mu\in F$ and one $t\in S$ such that $\alpha_t(\mu)=\sigma$, and such that $r(\mu)\in F$ for every $\mu\in F$.  

For a skew product $\Lambda\times_\eta S$, $F=\{(\lambda,1_S):\lambda\in \Lambda\}$ is a fundamental domain. The following ``Gross-Tucker Theorem'' says that existence of a fundamental domain characterises the actions $\lt$ on skew products. 

\begin{thm}\label{GTthm}
Suppose that $\Sigma$ is a row-finite $k$-graph with no sources, and $\alpha$ is a free action of an Ore semigroup $S$ on $\Sigma$ which admits a fundamental domain $F$. Let $q:\Sigma\to S\backslash \Sigma$ be the quotient map, and define $c:S\backslash \Sigma\to F$, $\eta:S\backslash \Sigma\to S$  and $\xi:\Sigma\to S$ by
\begin{equation}\label{defcetc*}
q(c(\lambda))=\lambda,\ \ s(c(\lambda))=\alpha_{\eta(\lambda)}(c(s(\lambda)))\ \text{ and }\ \sigma=\alpha_{\xi(\sigma)}(c(q(\sigma))).
\end{equation}
Then $\eta:S\backslash \Sigma\to S$ is a functor, and the map $\phi(\sigma):=(q(\sigma),\xi(\sigma))$ is an isomorphism of $\Sigma$ onto the skew product $(S\backslash \Sigma)\times_\eta S$, with inverse given by $\phi^{-1}(\lambda,t)=\alpha_t(c(\lambda))$. The isomorphism $\phi$ satisfies $\phi\circ\alpha_t=\lt_t\circ\phi$.
\end{thm}

When $S$ is a group, every free action of $S$ admits a fundamental domain, and we recover the result of \cite[Remark~5.6]{KP2}. Indeed, that proof starts by constructing a suitable fundamental domain. The rest of the proof of \cite[Remark~5.6]{KP2} then carries over to our situation, and shows that the formula we give for $\phi^{-1}$ defines an isomorphism of $(S\backslash \Sigma)\times_\eta S$ onto $\Sigma$.

\begin{example} \label{fdrem}
There are free semigroup actions which do not admit a fundamental domain. For example, consider  the $k$-graph $\Delta_k$ of \cite[\S 3]{KP3}, which has vertex set $\Delta_k^0 = \Z^k$, morphisms $\{(m,n) \in \mathbb{Z}^k \times \mathbb{Z}^k : m \le n \}$, $r (m,n) = m$,  $s (m,n) = n$, composition given by $(m,n)(n,p)=(m,p)$, and degree map $d:(m,n)\mapsto n-m$. There is a free action $\alpha$ of $\mathbb{N}^k$ on $\Delta_k$ such that $\alpha_p ( m , n ) = (m+p , n+p)$, and we claim that this action cannot have a fundamental domain. To see this, note that a fundamental domain $F$ would have to contain, for every $m\in \Delta_k^0$, a vertex $n\leq m$ (so that $m=\alpha_{m-n}(n)$ for some $n\in F$). Thus it would have to contain infinitely many vertices. But if $F$ has just two distinct vertices $n,p$, then every $m\geq n\vee p$ can be written as $m=\alpha_{m-n}(n)=\alpha_{m-p}(p)$. So there is no fundamental domain. 
\end{example}

\section{Dilating semigroup actions}\label{sec-dilate}

\begin{thm}[Laca]\label{laca}
Suppose that $S$ is an Ore semigroup with enveloping group $\Gamma=S^{-1}S$, and $\alpha:S\to \End A$ is an action of $S$ by injective extendible endomorphisms of a $C^*$-algebra $A$.

\smallskip
\textnormal{(a)} There are an action $\beta$ of $\Gamma$ on a $C^*$-algebra $B$ and an injective extendible homomorphism $j:A\to B$ such that
\begin{enumerate}
\item[\textnormal{(L1)}] $j\circ \alpha_u=\beta_u\circ j$ for $u\in S$, and

\smallskip
\item[\textnormal{(L2)}] $\bigcup_{u\in S}\beta_u^{-1}(j(A))$ is dense in $B$;
\end{enumerate}
the triple $(B,\beta, j)$ with these properties is unique up to isomorphism.

\smallskip
\textnormal{(b)} Suppose  $(B,\beta, j)$ has properties \textnormal{(L1)} and \textnormal{(L2)}, write $p:=\overline{i_B\circ j}(1)$, and define $v_s:=i_\Gamma(s)p$. Then $(i_B\circ j,v)$ is a covariant representation of $(A,S,\alpha)$, and $(i_B\circ j)\times v$ is an isomorphism of $A\times_\alpha S$ onto $p(B\rtimes_\beta \Gamma)p$.
\end{thm}

For the unital case, part (a) is Theorem~2.1 of \cite{Laca}. Laca proves the existence of $(B,\Gamma,\beta)$ using a direct-limit construction, and $j$ is the canonical embedding $\alpha^1$ of the first copy $A_1$ of $A$ in the direct limit $A_\infty$. Lemma~4.3 of \cite{PRY} says that if the endomorphisms are all extendible, then so is $j:=\alpha^1$. Laca's proof of uniqueness carries over verbatim. Part (b) is proved for the unital case in \cite[Theorem~2.4]{Laca}, and again the proof carries over: the crucial step, which is Lemma~2.3 of \cite{Laca}, is purely representation-theoretic.

In the context of graph algebras, Laca's theorem takes the following form.

\begin{cor}\label{corlaca}
Suppose that $S$ is an Ore semigroup with enveloping group $\Gamma=S^{-1}S$, and $\beta$ is a free action of $\Gamma$ on a row-finite $k$-graph $\Lambda$. Suppose that $\Omega$ is a saturated subgraph of $\Lambda$ such that $\beta_u(\Omega)\subset \Omega$ for all $u\in S$ and $\bigcup_{u\in S}\beta_u^{-1}(\Omega)=\Lambda$. Write $\alpha_u:=\beta_u|_\Omega$,  and set $p:=\overline{i_{C^*(\Lambda)}}(p_{\Omega^0})$. Then there is an isomorphism $\psi$ of $C^*(\Omega)\times_{\alpha_*}S$ onto $p(C^*(\Lambda)\rtimes_{\beta_*}\Gamma)p$ such that
\[
\psi(i_{C^*(\Omega)}(s^\Omega_\omega))=i_{C^*(\Lambda)}(s^\Lambda_\omega)\ \text{ and }\ \overline{\psi}(i_S(u))=i_\Gamma(u)p.
\]
\end{cor}

\begin{proof}
Let $\pi:\Omega\to\Lambda$ be the inclusion. Since $\Omega$ is saturated in $\Lambda$, Lemma~\ref{extendible} implies that $\pi$ induces an injective extendible homomorphism $\pi_*:C^*(\Omega)\to p_{\Omega^0}C^*(\Lambda)p_{\Omega^0}$ such that $\pi_*(s^\Omega_\omega)=s^\Lambda_\omega$ for $\omega\in\Omega$ and $\overline{\pi_*}(1)=p_{\Omega^0}$. Since each $\beta_u$ is an automorphism, it is saturated, and we claim that the restriction $\alpha_u$ is saturated as a graph morphism from $\Omega$ to $\Omega$. Indeed, if $\omega\in\Omega$ has $r(\omega)\in\alpha_u(\Omega^0)$, say $r(\omega)=\alpha_u(v)$, then
\[
r(\beta_u^{-1}(\omega))=\beta_u^{-1}(r(\omega))=\beta_u^{-1}(\alpha_u(v))=\beta_u^{-1}(\beta_u(v))=v
\]
belongs to $\Omega^0$, $\beta_u^{-1}(\omega)$ belongs to $\Omega$ because $\Omega$ is saturated in $\Lambda$, and $\omega=\alpha_u(\beta_u^{-1}(\omega))$ belongs to $\alpha_u(\Omega)$. Now Lemma~\ref{extendible} implies that $\alpha$ induces an action $\alpha_*$ of $S$ on $C^*(\Omega)$ by injective extendible endomorphisms.

We will show that the system $(C^*(\Lambda),\Gamma,\beta_*)$ and $j:=\pi_*$ have the properties (L1) and (L2) of Theorem~\ref{laca} relative to the semigroup dynamical system $(C^*(\Omega),S,\alpha_*)$. Homomorphisms are determined by what they do on generators, so for $\omega\in\Omega$ and $u\in S$, the calculation
\[
\pi_*((\alpha_*)_u(s^\Omega_\omega))= \pi_*(s^\Omega_{\alpha_u(\omega)})=s^\Lambda_{\alpha_u(\omega)}=s^\Lambda_{\beta_u(\omega)}=(\beta_*)_u(s^\Lambda_{\omega})=(\beta_*)_u(\pi_*(s^\Omega_{\omega}))
\]
implies that $\pi_*\circ(\alpha_*)_u=(\beta_*)_u\circ \pi_*$, which is (L1). Next, note that for $u\in S$ we have
\[
(\beta_*)_u^{-1}(\pi_*(C^*(\Omega)))\supset\big\{(\beta_*)_u^{-1}(s^\Lambda_\omega):\omega\in \Omega\big\}=\big\{s^\Lambda_{\beta_u^{-1}(\omega)}:\omega\in\Omega\big\},
\]
which by the hypothesis $\bigcup_{u\in S}\beta_u^{-1}(\Omega)=\Lambda$ implies that $A_0:=\bigcup_{u\in S}(\beta_*)_u^{-1}(\pi_*(C^*(\Omega)))$ contains all the generators of $C^*(\Lambda)$.   Thus to check (L2), it is enough to prove that $A_0$ is a $*$-algebra, and the only non-obvious point is whether $A_0$ is closed under multiplication. Let $a\in (\beta_*)_u^{-1}(\pi_*(C^*(\Omega)))$ and $b\in (\beta_*)_t^{-1}(\pi_*(C^*(\Omega)))$ for $u,t\in S$. Since $S$ is Ore, there exist $r,w\in S$ such that $ru=wt=x$, say. Since $(\beta_*)_r\circ\pi_*=\pi_*\circ(\alpha_*)_r$, we have $\range (\beta_*)_r\circ\pi_*\subset \range\pi_*$, and
\[
(\beta_*)_{u}^{-1}(\pi_*(C^*(\Omega)))=(\beta_*)_{ru}^{-1}\circ(\beta_*)_{r}(\pi_*(C^*(\Omega)))\subset(\beta_*)_{x}^{-1}(\pi_*(C^*(\Omega))).
\]
Similarly,
\[
(\beta_*)_{t}^{-1}(\pi_*(C^*(\Omega)))\subset(\beta_*)_{x}^{-1}(\pi_*(C^*(\Omega))).
\]
Since $(\beta_*)_{x}^{-1}(\pi_*(C^*(\Omega)))$ is an algebra, we have $ab\in (\beta_*)_{x}^{-1}(\pi_*(C^*(\Omega)))\subset A_0$, as required.

We can now set $v_s:=i_\Gamma(s)\overline{i_{C^*(\Lambda)}\circ\pi_*}(1)=i_\Gamma(s)p$, and deduce from Theorem~\ref{laca} that $\psi:=(i_{C^*(\Lambda)}\circ \pi_*)\times v$ is an isomorphism of $C^*(\Omega)\times_{\alpha_*}S$ onto $p(C^*(\Lambda)\rtimes_{\beta_*}\Gamma)p$. This isomorphism has the required properties.
\end{proof}

\section{Crossed products of the $C^*$-algebras of skew-product graphs}\label{sec-group}

The action $\lt$ of a group $\Gamma$ on a skew-product $\Lambda\times_\eta\Gamma$ induces an action $\lt_*$ of $\Gamma$ on the graph algebra $C^*(\Lambda\times_\eta \Gamma)$. Kumjian and Pask proved in \cite{KP1} that the crossed product by this action is stably isomorphic to $C^*(\Lambda)$. Their proof used a groupoid model for the graph algebra and results of Renault about skew-product groupoids, and an explicit isomorphism was constructed in \cite{KQR}. In the following generalisation of \cite[Theorem~3.1]{KQR}, the existence of an isomorphism follows from \cite[Theorem~5.7]{KP2} or \cite[Corollary~5.1]{PQR} (taking $H=G$), but we want an explicit isomorphism.

\begin{thm}\label{cpofskewp}
Suppose that $\Lambda$ is a row-finite $k$ graph with no sources, and $\eta:\Lambda\to \Gamma$ is a functor into a group $\Gamma$. Then there is an isomorphism $\phi$ of $C^*(\Lambda \times_\eta \Gamma)\rtimes_{\lt_*}\Gamma$ onto $C^*(\Lambda)\otimes\K(l^2(\Gamma))$ such that
\begin{equation}\label{defphi}
\phi(i_{C^*(\Lambda \times_\eta \Gamma)}(s_{(\lambda,g)}))=s_{\lambda}\otimes \chi_g\rho_{\eta(\lambda)}\ \text{ and }\ \overline{\phi}(i_\Gamma(h))=1\otimes \lambda_h.
\end{equation}
\end{thm}

We first show the existence of the homomorphism $\phi$. To do this, we verify the following statements in $C^*(\Lambda)\otimes \K(l^2(\Gamma))$:

\begin{enumerate}
\item\label{CR1} $S_{(\lambda,g)}:=s_{\lambda}\otimes \chi_g\rho_{\eta(\lambda)}$ is a Cuntz-Krieger $(\Lambda \times_\eta \Gamma)$-family;
\smallskip

\item\label{CR2} if $F_n$ and $G_n$ are increasing sequences of finite subsets of $\Lambda^0$ and $\Gamma$  such that $\Lambda^0=\bigcup_{n}F_n$ and $\Gamma=\bigcup_n G_n$, then $\sum_{(v,g)\in F_n\times G_n}S_{(v,g)}$ converges strictly to $1$;

\smallskip
\item\label{CR3} $(1\otimes\lambda_h)S_{(\lambda,g)}=S_{(\lambda,hg)}(1\otimes\lambda_h)$.
\end{enumerate}

To check (CK1) for the family in \eqref{CR1}, we take $(v,g)$ and $(w,h)$ in $(\Lambda\times_\eta \Gamma)^0=\Lambda^0\times\Gamma$: then $\eta(v)=\eta(w)=1$, and $S_{(v,g)}S_{(w,h)}=s_vs_w\otimes \chi_g\chi_h$, which gives (CK1). Next, suppose that $(\lambda,g)$ and $(\mu,h)$ are composable, so that $s(\lambda)=r(\mu)$ and $g\eta(\lambda)=h$. Then $\rho_k\chi_{gk}=\chi_g\rho_k$ implies that
\begin{align*}
S_{(\lambda,g)}S_{(\mu,h)}&=(s_\lambda\otimes \chi_g\rho_{\eta(\lambda)})(s_\mu\otimes \chi_h\rho_{\eta(\mu)})=(s_\lambda s_\mu)\otimes (\chi_g\rho_{\eta(\lambda)}\chi_{g\eta(\lambda)}\rho_{\eta(\mu)})\\
&=s_{\lambda\mu}\otimes(\chi_g\chi_g\rho_{\eta(\lambda)}\rho_{\eta(\mu)})=s_{\lambda\mu}\otimes(\chi_g\rho_{\eta(\lambda\mu)})=S_{(\lambda\mu,g)},
\end{align*}
which is (CK2). A similar calculation gives (CK3), and a calculation using the Cuntz-Krieger relation for $\{s_\lambda\}$ gives (CK4). We have now proved item \eqref{CR1}.

Next observe that
\[
\sum_{(v,g)\in F_n\times G_n}S_{(v,g)}=\Big(\sum_{v\in F_n} s_v\Big)\otimes\Big(\sum_{g\in G_n} \chi_g\Big)=a_n\otimes b_n,
\]
say, and then \eqref{CR2} holds because $\{a_n\}$ and $\{b_n\}$ are approximate identities for $C^*(\Lambda)$ and $\K(l^2(\Gamma))$. Finally, a calculation using $\lambda_h\chi_g=\chi_{hg}\lambda_h$ and $\rho_g\lambda_h=\lambda_h\rho_g$ gives \eqref{CR3}.

Item \eqref{CR1} implies that there is a homomorphism $\pi_S$ from $C^*(\Lambda \times_\eta \Gamma)$ to $C^*(\Lambda)\otimes \K(l^2(\Gamma))$ taking $s_{(\lambda,g)}$ to $S_{(\lambda,g)}$, and \eqref{CR2} then says that $\pi_S$ is nondegenerate. Item \eqref{CR3} implies that $(\pi_S,1\otimes\lambda)$ is a covariant representation of $(C^*(\Lambda \times_\eta \Gamma),\Gamma,\lt_*)$ in $C^*(\Lambda)\otimes \K(l^2(\Gamma))$, and $\phi:=\pi_S\times(1\otimes\lambda)$ satisfies \eqref{defphi}. The image of each spanning element $s_{(\lambda,g)}s_{(\mu,k)}^*i_\Gamma(h)$ belongs to $C^*(\Lambda)\otimes \K(l^2(\Gamma))$, and hence $\phi$ has range in $C^*(\Lambda)\otimes \K(l^2(\Gamma))$.

To see that $\phi$ is surjective, we note that the range of $\phi$ contains every element $s_\lambda\otimes \chi_g\rho_{\eta(\lambda)}\lambda_h$. The operator $\chi_g\rho_{\eta(\lambda)}\lambda_h$ is the rank-one operator $e_g\otimes\overline{e}_{h^{-1}g\eta(\lambda)}$, and for each $\lambda$, each matrix unit $e_p\otimes\overline{e}_q$ arises for a suitable choice of $g$ and $h$. Thus the range of $\phi$ contains every  $s_\lambda\otimes(e_p\otimes\overline{e}_q)$, and every
\[
s_\lambda s_\mu^*\otimes(e_p\otimes\overline{e}_q)=\big(s_\lambda\otimes(e_p\otimes\overline{e}_q)\big)\big(s_\mu\otimes(e_q\otimes\overline{e}_q)\big)^*;
\]
since these elements span a dense $*$-subalgebra of $C^*(\Lambda)\otimes\K(l^2(\Gamma))$, and homomorphisms of $C^*$-algebras have closed range, we deduce that $\phi$ is surjective.

To prove that $\phi$ is injective, we will construct a left inverse for $\phi$. Recall that if $A$ is a $C^*$-algebra then $UA$ denotes the group of unitary elements of $A$.

\begin{lemma}\label{defytimesu}
Suppose that $\{y_g:g\in\Gamma\}$ is a set of mutually orthogonal projections in a $C^*$-algebra $D$, and $u:\Gamma\to UM(D)$ is a homomorphism such that
\begin{equation}\label{commrel1}
u_hy_g=y_{hg}u_h.
\end{equation}
Then there is a homomorphism $y\times u:\K(l^2(\Gamma))\to D$ such that $y\times u(\lambda_h\chi_g)=u_hy_g$.
\end{lemma}

\begin{proof}
Observe that $e_{g,h}:=u_gy_1u_h^*$ is a set of matrix units in $D$, and thus Corollary~A.9 of \cite{cbms} gives a homomorphism $y\times u:\K(l^2(\Gamma))\to D$ such that $(y\times u)(e_g\otimes\overline e_h)=u_gy_1u_h^*$. Now verify that $\lambda_h\chi_g=e_{hg}\otimes \overline e_g$, and we have $y\times u(\lambda_h\chi_g)=u_{hg}y_1u_g^*=u_hy_g$.
\end{proof}

\begin{lemma}\label{eltsforinv}
Suppose that $\Lambda$, $\Gamma$ and $\eta$ are as in Theorem~\ref{cpofskewp}.

\smallskip
\textnormal{(a)} The elements
\[
y_g:=\overline{i_{C^*(\Lambda\times_\eta\Gamma)}}(p_{\Lambda^0\times\{g\}})\ \text{ and }\ u_h:=i_\Gamma(h)
\]
of $M(C^*(\Lambda \times_\eta \Gamma)\rtimes_{\lt_*}\Gamma)$ satisfy \eqref{commrel1}. The homomorphism $y\times u$ from Lemma~\ref{defytimesu} is nondegenerate; the elements $w_k:=\overline{y\times u}(\rho_k)$ commute with $u_h$ and satisfy
\begin{equation}\label{commrel2}
w_ky_g=y_{gk^{-1}}w_k.
\end{equation}

\textnormal{(b)} The partial isometries
\[
T_\lambda:=\overline{i_{C^*(\Lambda\times_\eta\Gamma)}}(s_{\{\lambda\}\times\Gamma})w_{\eta(\lambda)}^{-1}
\]
commute with every $y_g$, $u_h$ and $w_k$.

\smallskip
\textnormal{(c)}  $\{T_\lambda:\lambda\in\Lambda\}$ is a Cuntz-Krieger $\Lambda$-family in $M(C^*(\Lambda \times_\eta \Gamma)\rtimes_{\lt_*}\Gamma)$.
\end{lemma}

\begin{proof}
We choose increasing sequences of finite subsets $G_n$ of $\Lambda^0$ and $H_n$ of $\Gamma$ such that $\Lambda^0=\bigcup_n G_n$ and $\Gamma=\bigcup_n H_n$. Then the strict continuity of $\overline{i_{C^*(\Lambda\times_\eta\Gamma)}}$ implies that
\[
i_{C^*(\Lambda\times_\eta\Gamma)}\Big(\sum_{v\in G_n}s_{(v,g)}\Big)\to y_g\ \text{ strictly.}
\]
For each $n$, the covariance of $(i_{C^*(\Lambda\times_\eta\Gamma)},i_\Gamma)$ implies that
\[
u_hi_{C^*(\Lambda\times_\eta\Gamma)}\Big(\sum_{v\in G_n}s_{(v,g)}\Big)=i_{C^*(\Lambda\times_\eta\Gamma)}\Big(\sum_{v\in G_n}s_{(v,hg)}\Big)u_h,
\]
and since the right-hand side converges strictly to $y_{hg}u_h$, \eqref{commrel1} follows.

Since $r(\alpha,g)=(r(\alpha),g)$ belongs to $\Lambda^0\times\{g\}$, the formula \eqref{leftmult} shows that $s_{(\alpha,g)}s_{(\beta,h)}^*=y_gs_{(\alpha,g)}s_{(\beta,h)}^*$, and this implies that $y\times u$ is nondegenerate. So the formula for $w_k$ makes sense. It has the described properties because $\rho_k$ commutes with $\lambda_h$ and satisfies $\rho_k\chi_g=\chi_{gk^{-1}}\rho_k$. We have now proved (a).

The last assertion in Lemma~\ref{defsumpis} implies that
\begin{equation}\label{ygTmu}
y_gT_\lambda=\overline{i_{C^*(\Lambda\times_\eta\Gamma)}}(p_{\Lambda^0\times\{g\}}s_{\{\lambda\}\times \Gamma})w_{\eta(\lambda)}^{-1}=i_{C^*(\Lambda\times_\eta\Gamma)}(s_{(\lambda,g)})w_{\eta(\lambda)}^{-1}.
\end{equation}
On the other hand, \eqref{commrel2} implies that $w_{\eta(\lambda)}^{-1}y_g=y_{g\eta(\lambda)}w_{\eta(\lambda)}^{-1}$, and thus
\[
T_\lambda y_g=\overline{i_{C^*(\Lambda\times_\eta\Gamma)}}(s_{\{\lambda\}\times \Gamma}p_{\Lambda^0\times\{g\eta(\lambda)\}})w_{\eta(\lambda)}^{-1},
\]
which since $s(\lambda,g)=(s(\lambda),g\eta(\lambda))$ is the same as the right-hand side of \eqref{ygTmu}. Thus $y_g$ commutes with $T_\lambda$.

To see that $u_h$ commutes with $T_\lambda$, we realise $s_{\{\lambda\}\times\Gamma}$ as the strict limit of the finite sums $s_{\{\lambda\}\times H_n}:=\sum_{g\in H_n}s_{(\lambda,g)}$. Then $T_\lambda$ is the strict limit of $t_n:=i_{C^*(\Lambda\times_\eta\Gamma)}(s_{\{\lambda\}\times H_n})$, and $u_hT_\lambda$ is the strict limit of $u_ht_n$. Covariance implies that
\begin{equation}\label{uhTmu}
u_ht_n=i_\Gamma(h)i_{C^*(\Lambda\times_\eta\Gamma)}(s_{\{\lambda\}\times H_n})=
i_{C^*(\Lambda\times_\eta\Gamma)}(s_{\{\lambda\}\times hH_n})u_h,
\end{equation}
and since the limit $s_{\{\lambda\}\times\Gamma}$ is independent of the choice of increasing subsets, the right-hand side of \eqref{uhTmu} converges strictly to $T_\lambda u_h$. Thus $u_hT_\lambda=T_\lambda u_h$. Since $T_\lambda$ commutes with everything in the ranges of $y\times u$ and $\overline{y\times u}$, including $w_k$, we have proved (b).

Since $\eta(v)=1$ for every vertex $v$, the relation (CK1) for $\{T_\lambda\}$ follows from the assertion $s_Vs_W=s_{VW}$ in Lemma~\ref{defsumpis}. For (CK2), we suppose $\lambda$ and $\mu$ are composable in $\Lambda$. Then because  $w_{\eta(\lambda)}^{-1}$ and $T_\mu$ commute, we have
\[
T_\lambda T_\mu=\overline{i_{C^*(\Lambda\times_\eta\Gamma)}}(s_{\{\lambda\}\times\Gamma})T_\mu w_{\eta(\lambda)}^{-1}=\overline{i_{C^*(\Lambda\times_\eta\Gamma)}}(s_{\{\lambda\}\times\Gamma}s_{\{\mu\}\times\Gamma})(w_{\eta(\lambda)\eta(\mu)})^{-1},
\]
and the right-hand side reduces to $T_{\lambda\mu}$ because $s_Vs_W=s_{VW}$, $k\mapsto w_k$ is a homomorphism, and $\eta$ is a functor. For (CK3), we need to compute
\[
T_\lambda^*T_\lambda=w_{\eta(\lambda)}\overline{i_{C^*(\Lambda\times_\eta\Gamma)}}(s_{\{\lambda\}\times\Gamma}^*s_{\{\lambda\}\times\Gamma})w_{\eta(\lambda)}^{-1}.
\]
From \eqref{leftmult} and the adjoint of \eqref{rightmult}, we deduce that $(s_{\{\lambda\}\times\Gamma}^*s_{\{\lambda\}\times\Gamma})(s_{(\alpha,g)} s_{(\beta,h)}^*)$ vanishes unless $r(\alpha)=s(\lambda)$, and then is $s_{(\alpha,g)} s_{(\beta,h)}^*$; thus left multiplication by $s_{\{\lambda\}\times\Gamma}^*s_{\{\lambda\}\times\Gamma}$ is the same as left multiplication by $s_{\{s(\lambda)\}\times\Gamma}$, and $s_{\{\lambda\}\times\Gamma}^*s_{\{\lambda\}\times\Gamma}=s_{\{s(\lambda)\}\times\Gamma}$. Thus $T_\lambda^* T_\lambda=w_{\eta(\lambda)}T_{s(\lambda)}w_{\eta(\lambda)}^{-1}$, and since $w_{\eta(\lambda)}$ commutes with $T_{s(\lambda)}$, we recover (CK3). For (CK4) we fix $v\in \Lambda^0$ and $n\in\N^k$, and then
\[
\sum_{\lambda\in v\Lambda^n}T_\lambda T_\lambda^*=\overline{i_{C^*(\Lambda\times_\eta\Gamma)}}\Big(\sum_{\lambda\in v\Lambda^n}s_{\{\lambda\}\times\Lambda}s_{\{\lambda\}\times\Lambda}^*\Big);
\]
a calculation using the formulas in Lemma~\ref{defsumpis}  shows that left multiplication by the inside sum is the same as left multiplication by $s_{\{v\}\times\Gamma}$, and this gives (CK4).
\end{proof}

\begin{proof}[Proof of Theorem~\ref{cpofskewp}]
In the paragraphs following the statement, we constructed $\phi$ and showed it is surjective. For injectivity, we consider the homomorphism
$y\times u:\K(l^2(\Gamma))\to M(C^*(\Lambda \times_\eta \Gamma)\rtimes_{\lt_*}\Gamma)$ associated to the elements $y_g$ and $u_h$ described in Lemma~\ref{eltsforinv}(a), and the homomorphism $\pi_T$ of $C^*(\Lambda)$ into $M(C^*(\Lambda \times_\eta \Gamma)\rtimes_{\lt_*}\Gamma)$ given by the Cuntz-Krieger family $\{T_\lambda\}$ of Lemma~\ref{eltsforinv}. Lemma~\ref{eltsforinv}(b) implies that $\pi_T$ and $y\times u$ have commuting ranges, and hence give a homomorphism  $\theta:=\pi_T\otimes(y\times u)$ of $C^*(\Lambda)\otimes \K(l^2(\Gamma))$ into $M(C^*(\Lambda \times_\eta \Gamma)\rtimes_{\lt_*}\Gamma)$ such that $\theta(a\otimes k)=\pi_T(a)(y\times u)(k)$ (by \cite[Theorem~B.27]{RW}).

Finally we compute, using in particular the formula \eqref{commrel2}:
\begin{align*}
\theta\circ\phi\big(&i_{C^*(\Lambda\times_\eta\Gamma)}(s_{(\lambda,g)})i_\Gamma(h)\big)
=\theta\big(s_\lambda\otimes \chi_g\rho_{\eta(\lambda)}\lambda_h\big)\\
&=\overline{i_{C^*(\Lambda\times_\eta\Gamma)}}(s_{\{\lambda\}\times\Gamma})w_{\eta(\lambda)}^{-1}y_gw_{\eta(\lambda)}u_h=\overline{i_{C^*(\Lambda\times_\eta\Gamma)}}(s_{\{\lambda\}\times\Gamma})
y_{g\eta(\lambda)}w_{\eta(\lambda)}^{-1}w_{\eta(\lambda)}u_h\\
&=\overline{i_{C^*(\Lambda\times_\eta\Gamma)}}\big(s_{\{\lambda\}\times\Gamma}p_{\Lambda^0\times\{g\eta(\lambda)\}}\big)i_\Gamma(h)=i_{C^*(\Lambda\times_\eta\Gamma)}(s_{(\lambda,g)})i_\Gamma(h).
\end{align*}
Since the elements $i_{C^*(\Lambda\times_\eta\Gamma)}(s_{(\lambda,g)})i_\Gamma(h)$ generate the crossed product, this proves that $\theta\circ\phi$ is the identity, and in particular that $\phi$ is injective.
\end{proof}

\section{The main theorem}\label{sec-main}

\begin{thm}\label{mainthm}
Suppose that $\Sigma$ is a row-finite $k$-graph with no sources, and $\alpha$ is a free action of an Ore semigroup $S$ on $\Sigma$ which admits a fundamental domain $F$. Let $q:\Sigma\to S\backslash \Sigma$ be the quotient map, and define $c:S\backslash \Sigma\to F$, $\eta:S\backslash \Sigma\to S$, $\xi:\Sigma\to S$ by
\begin{equation}\label{defcetc}
q(c(\lambda))=\lambda,\ \ s(c(\lambda))=\alpha_{\eta(\lambda)}(c(s(\lambda)))\ \text{ and }\ \sigma=\alpha_{\xi(\sigma)}(c(q(\sigma))).
\end{equation}
Then there is an isomorphism $\psi$ of $C^*(\Sigma)\times_{\alpha_*}S$ onto $C^*(S\backslash \Sigma)\otimes \K(l^2(S))$ such that
\[
\psi(i_{C^*(\Sigma)}(s^\Sigma_\sigma))=s_{q(\sigma)}\otimes(\chi_{\xi(\sigma)}\rho_{\eta(q(\sigma))}|_{l^2(S)})\ \text{ and }\ \overline{\psi}(i_S(u))=1\otimes \lambda^S_u.
\]
\end{thm}

We need a general lemma about tensor products of multipliers.

\begin{lemma}\label{multalgotimes}
Suppose that $A$ and $B$ are $C^*$-algebras. For each $m\in M(A)$ and $n\in M(B)$ there is a multiplier $m\otimes_{\max} n$ of $A\otimes_{\max} B$ such that
\begin{equation}\label{deftp}
(m\otimes_{\max} n)(a\otimes b)=ma\otimes nb\ \text{ and }\ (a\otimes b)(m\otimes_{\max} n)=am\otimes bn.
\end{equation}
The map $\iota:(m,n)\mapsto m\otimes_{\max} n$ is strictly continuous in the following weak sense: if $m_i\to m$ strictly in $M(A)$, $n_i\to n$ strictly in $M(B)$, and both $\{m_i\}$ and $\{n_i\}$ are bounded, then $m_i\otimes_{\max}n_i\to m\otimes_{\max} n$ strictly.
\end{lemma}

\begin{proof}
Consider the canonical maps $j_A:A\to M(A\otimes_{\max}B)$ and $j_B:B\to M(A\otimes_{\max}B)$, as in, for example, \cite[Theorem~B.27]{RW}. Then $j_A$ and $j_B$ are nondegenerate homomorphisms with commuting ranges such that $j_A(a)j_B(b)=a\otimes b$ \cite[Theorem~B.27(a)]{RW}. The extensions $\overline{j_A}$ to $M(A)$ and $\overline{j_B}$ to $M(B)$ also have commuting ranges, and hence there is a homomorphism $\overline{j_A}\otimes_{\max}\overline{j_B}$ of $M(A)\otimes_{\max}M(B)$ into $M(A\otimes_{\max}B)$ such that $\overline{j_A}\otimes_{\max}\overline{j_B}(m\otimes n)=\overline{j_A}(m)\overline{j_B}(n)$. We define $m\otimes_{\max}n:=\overline{j_A}\otimes_{\max}\overline{j_B}(m\otimes n)$ Then
\begin{align*}
(m\otimes_{\max}n)(a\otimes b)&=(\overline{j_A}(m)\overline{j_B}(n))(i_A(a)i_B(b))=(\overline{j_A}(m)j_A(a))(\overline{j_B}(n)j_B(b))\\
&=j_A(ma)j_B(nb)=ma\otimes nb,
\end{align*}
and similarly on the other side. Since $\overline{j_A}$ and $\overline{j_B}$ are strictly continuous, $\overline{j_A}(m_i)\to \overline{j_A}(m)$ and $\overline{j_B}(n_i)\to \overline{j_B}(n)$, and the strict continuity of multiplication on bounded sets implies that $m_i\otimes_{\max}n_i=\overline{j_A}(m_i)\overline{j_B}(n_i)$ converges to $\overline{j_A}(m)\overline{j_B}(n)=m\otimes_{\max} n$.
\end{proof}

\begin{remark}
When we apply Lemma~\ref{multalgotimes}, at least one of $A$ or $B$ is nuclear, and $A\otimes_{\max} B$ coincides with the usual spatial tensor product; then,  since there is at most one multiplier satisfying \eqref{deftp}, $m\otimes_{\max} n$ coincides with the usual spatially defined $m\otimes n$. However, $M(A)$ and $M(B)$ need not be nuclear (even for $B=\K(\H)$!), so this observation merely says that $\overline{j_A}\otimes_{\max}\overline{j_B}$ on $M(A)\otimes_{\max} M(B)$ factors through the spatial tensor product.
\end{remark}

\begin{proof}[Proof of Theorem~\ref{mainthm}]
Our Gross-Tucker theorem (Theorem~\ref{GTthm}) describes an isomorphism $\phi$ of $\Sigma$ onto the skew product $(S\backslash\Sigma)\times_\eta S$ such that $\phi\circ\alpha_t={\lt_t}\circ{\phi}$. The induced isomorphism $\phi_*$ of $C^*(\Sigma)$ onto $C^*((S\backslash\Sigma)\times_\eta S)$ satisfies $\phi_*\circ\alpha_*={\lt_*}\circ \phi_*$, and hence induces an isomorphism $\psi_1$ of $C^*(\Sigma)\times_{\alpha_*}S$ onto $C^*((S\backslash\Sigma)\times_\eta S)\times_{\lt_*} S$ satisfying
\[
\psi_1(i_{C^*(\Sigma)}(s^\Sigma_\sigma))=i_{C^*((S\backslash\Sigma)\times_\eta S)}(s_{(q(\sigma),\xi(\sigma))})\ \text{ and }\ \overline{\psi_1}(i_S(u))=i_S(u).
\]

We want to apply Corollary~\ref{corlaca} with $\Lambda=(S\backslash\Sigma)\times_\eta \Gamma$, $\Omega=(S\backslash\Sigma)\times_\eta S$ and $\beta=\lt$. The subgraph $\Omega$ is saturated, because $r(\lambda,g)=(r(\lambda),g)$ belongs to $\Omega^0$ precisely when $g\in S$, in which case $(\lambda,g)$ belongs to $\Omega$. We trivially have $\lt_t(\Omega)\subset\Omega$ for $t\in S$, and because $\Gamma=S^{-1}S$, every $g\in \Gamma$ can be written as $t^{-1}u$ for $t,u\in S$, and then every $(\lambda,g)=\lt_t^{-1}(\lambda,u)$ belongs to $\bigcup_{t\in S}\lt_t^{-1}(\Omega)$. The restriction of $\lt_u$ to $\Omega$ is just the $\lt_u$ in the previous paragraph. So with $p:=\overline{i_{C^*((S\backslash\Sigma)\times_\eta S)}}(p_{(S\backslash\Sigma)^0\times S})$,
Corollary~\ref{corlaca} gives an isomorphism $\psi_2$ of $C^*((S\backslash\Sigma)\times_\eta S)\times_{\lt_*} S$ onto $p\big(C^*((S\backslash\Sigma)\times_\eta \Gamma)\rtimes_{\lt_*} \Gamma\big)p$ such that
\begin{equation}\label{defpsi2}
\psi_2\big(i_{C^*((S\backslash\Sigma)\times_\eta S)}(s_{(\lambda,t)})\big)=
i_{C^*((S\backslash\Sigma)\times_\eta \Gamma)}(s_{(\lambda,t)})\ \text{ and }\ \overline{\psi_2}(i_S(u))=i_\Gamma(u)p.
\end{equation}

Theorem~\ref{cpofskewp} gives an isomorphism $\phi$ of $C^*((S\backslash\Sigma)\times_\eta \Gamma)\rtimes_{\lt_*} \Gamma$ onto $C^*(S\backslash\Sigma)\otimes\K(l^2(\Gamma))$ such that
\begin{equation*}
\phi\big(i_{C^*((S\backslash\Sigma) \times_\eta \Gamma)}(s_{(\lambda,g)})\big)=s_{\lambda}\otimes \chi_g\rho_{\eta(\lambda)}\ \text{ and }\ \overline{\phi}(i_\Gamma(h))=1\otimes \lambda_h.
\end{equation*}
Since $\phi$ is an isomorphism, it extends to the multiplier algebra, and restricts an isomorphism of $p\big(C^*((S\backslash\Sigma)\times_\eta \Gamma)\rtimes_{\lt_*} \Gamma\big)p$ onto $\overline{\phi}(p)\big(C^*((S\backslash\Sigma)\times_\eta \Gamma)\rtimes_{\lt_*} \Gamma\big)\overline{\phi}(p)$.

Again write $(S\backslash\Sigma)^0$ and $S$ as increasing unions $\bigcup_n G_n$ and $S=\bigcup_n H_n$ of finite subsets. Then $p_{(S\backslash\Sigma)^0\times S}$ is by definition the strict limit of $p_n:=\sum_{(w,u)\in G_n\times H_n}p_{(w,u)}$ (see Lemma~\ref{defsumpis}). Thus, since $\eta(w)=1$ for every vertex $w$, $\overline{\phi}(p)$ is the strict limit of
\begin{align*}
\phi\big(i_{C^*((S\backslash\Sigma)\times_\eta \Gamma)}(p_n)\big)&=\sum_{(w,u)\in G_n\times H_n}\phi\big(i_{C^*((S\backslash\Sigma)\times_\eta \Gamma)}(p_{(w,u)})\big)\\
&=\sum_{(w,u)\in G_n\times H_n}p_w\otimes \chi_u\\
&=\Big(\sum_{w\in G_n} p_w\Big)\otimes\Big(\sum_{u\in H_n}\chi_u\Big).
\end{align*}
Since $\sum_{w\in G_n} p_w$ and $\sum_{u\in H_n}\chi_u$ converge strictly to $1_{M(C^*(S\backslash\Sigma))}$ and $\chi_S$, the assertion about strict continuity in Lemma~\ref{multalgotimes} implies that $\overline{\phi}(p)=1_{M(C^*(S\backslash\Sigma))}\otimes\chi_S$. A calculation on elementary tensors shows that
\[
(1\otimes\chi_S)(C^*(S\backslash\Sigma)\otimes \K(l^2(\Gamma))(1\otimes\chi_S)
=C^*(S\backslash\Sigma)\otimes \chi_S\K(l^2(\Gamma))\chi_S.
\]
Since we are identifying $l^2(S)$ with a subspace of $l^2(\Gamma)$ and $\chi_S$ is then the orthogonal projection of $l^2(\Gamma)$ onto $l^2(S)$, $\chi_S\K(l^2(\Gamma)\chi_S$ is naturally identified with $\K(l^2(S))$. When we make this identification, $\chi_S\lambda_u\chi_S$ is the generator $\lambda_u^S:=\lambda_u\chi_S$ of the Toeplitz representation of $S$ on $l^2(S)$. Thus restricting $\phi$ gives an isomorphism
\[
\psi_3:p\big(C^*((S\backslash\Sigma)\times_\eta \Gamma)\rtimes_{\lt_*} \Gamma\big)p\to  C^*(S\backslash\Sigma)\otimes \K(l^2(S))
\]
such that for $t$ and $u$ in $S$,
\[
\psi_3\big(p\,i_{C^*((S\backslash\Sigma)\times_\eta \Gamma)}(s_{(\lambda,t)})p\big)=s_\lambda\otimes (\chi_t\rho_{\eta(\lambda)})|_{l^2(S)}\ \text{ and }\ \overline{\psi_3}(i_\Gamma(u))=\lambda^S_u
\]
(notice that although $\rho_{\eta(\lambda)}$ does not leave $l^2(S)$ invariant, the product $\chi_t\rho_{\eta(\lambda)}$ does).

Now $\psi:=\psi_3\circ\psi_2\circ\psi_1$ has the required properties.
\end{proof}

\begin{cor} \label{dypi}
Suppose that $\Sigma$ is a $2$-graph, $\alpha$ is a free action of an Ore semigroup $S$ on $\Sigma$ which admits a fundamental domain $F$. Then $C^* ( \Sigma ) \times_{\alpha_*} S$ is purely infinite and simple if and only if $C^*(S\backslash \Sigma)$ is purely infinite and simple.
\end{cor}

\begin{proof}
Both simplicity and pure-infiniteness are preserved by stable isomorphism (by \cite[Proposition~4.1.8]{Ror} for pure infiniteness), so the result follows from Theorem~\ref{mainthm}.
\end{proof}

The point of the Corollary is that $S\backslash \Sigma$ is smaller than $\Sigma$, hence is likely to be more tractable, and we have criteria for deciding whether $C^*(S\backslash \Sigma)$ is purely infinite and simple. We illustrate with a example which is similar to one studied in \cite{KPa}.

\begin{example} \label{ex:app} We consider the graph $\mathbb{F}_\theta^2$ of \cite{KP, DY} associated to the permutation $\theta$ of $\{1,2,3\}\times \{1,2,3\}$
defined by
\[
\begin{array}{l}
\theta ( 2,j ) = (1,j),\ \theta ( 1,j ) = (2,j),\ \theta (3,j) = (3,j) \text{ for } j=1,3, \text{ and } \\
\theta ( i,2 ) = (i,2) ~ \text{ for } i =1,2,3.
\end{array}
\]
As in \cite[Example~5.7]{KPa}, there is a functor $c: \mathbb{F}_\theta^2 \rightarrow \mathbb{Z}^2$ such that $c(g_3)=(0,1)$, $c(f_3)=(1,0)$, and $c(f_i)=(0,0)$, $c(g_i)=(0,0)$ for $i=1,2$. Since this functor takes values in $\N^2$, we can apply Corollary~\ref{dypi} to the action $\lt$ of $\N^2$ on $\mathbb{F}_\theta^2\times_c\N^2$, for which the quotient graph is $\mathbb{F}_\theta^2$. It is shown in \cite[Example 5.7]{KPa} that $\mathbb{F}_\theta^2$ is aperiodic, and since $\mathbb{F}_\theta^2$ has a single vertex, it is trivially cofinal. Thus $C^*(\mathbb{F}_\theta^2)$ is simple by \cite[Theorem 3.4]{RobS}, and purely infinite by \cite[Proposition 8.8]{S}. Thus Corollary~\ref{dypi} implies that $C^* ( \mathbb{F}_\theta^2 \times_c \mathbb{N}^2 ) \times_{\lt_*} \mathbb{N}^2$ is purely infinite and simple. On the other hand, the discussion in \cite[Example~3.5]{KPa} shows that $C^* ( \mathbb{F}_\theta^2 \times_c \mathbb{N}^2 )$ has many ideals.
\end{example}

\end{document}